\documentclass{amsart}
\usepackage{amsfonts}
\usepackage{graphicx}
\usepackage{tabularx}
\usepackage{array}
\usepackage[usenames,dvipsnames]{color}
\usepackage{comment}
\usepackage{amsmath}
\usepackage{amsthm}
\usepackage{amssymb}
\usepackage{fullpage}
\usepackage{xcolor}
\usepackage{tikz}

\tikzstyle{legend_general}=[rectangle, rounded corners, thin,
                          top color= white,bottom color=lavander!25,
                          minimum width=2.5cm, minimum height=0.8cm,
                          violet]

\DeclareMathOperator{\col}{col}

\newtheorem{theorem}{Theorem}[section]

\newtheorem{claim}[theorem]{Claim}
\newtheorem{lemma}[theorem]{Lemma}

\theoremstyle{definition}

\title{A note on the connected game coloring number} 

\author{Peter Bradshaw}
\address{Department of Mathematics, Simon Fraser University, Vancouver, Canada}
\email{pabradsh@sfu.ca}

\thanks{The author of this work has been partially supported by a supervisor's grant from the Natural Sciences and Engineering Research Council of Canada (NSERC)}
\begin{document}
\begin{abstract}
We consider the \emph{connected game coloring number} of a graph, introduced by Charpentier et al. as a game theoretic graph parameter that measures the degeneracy of a graph with respect to a connected version of the graph marking game. We establish bounds for the connected game coloring number of graphs of bounded treedepth and of $k$-trees. We also show that there exists an outerplanar $2$-tree with connected game coloring number $5$, which answers a question from [C. Charpentier, H. Hocquard, E. Sopena, and X. Zhu. A connected version of the graph coloring game. \textit{Discrete Appl. Math.}, 2020].
\end{abstract}
\maketitle
\section{Introduction}
Let $G$ be a finite graph. The \emph{coloring game} on $G$, introduced by Bodlaender \cite{Bodlaender}, is defined as follows. Two players, Alice and Bob, take turns coloring vertices of $G$ using colors from a set of $k$ colors. On each player's move, the player chooses an uncolored vertex from $V(G)$ and colors this vertex with one of the $k$ given colors. Alice moves first, and on each turn, both players are forbidden from coloring a vertex $v$ with a color already used at a neighbor of $v$; that is, both players must color $G$ properly. The game continues until all of $G$ is colored or there is no legal move. Alice's goal is to complete a proper coloring on $G$, and Bob's goal is to reach a game state in which $G$ has at least one uncolored vertex but no move is legal. The \emph{game chromatic number} of $G$, written $\chi_g(G)$, is defined as the minimum integer $k$ for which Alice has a winning strategy in the coloring game played with $k$ colors. The game chromatic number is related to the chromatic number by the following relation:
$$\chi(G) \leq \chi_g(G) \leq \Delta(G) + 1.$$
Informally, the game chromatic number gives a measure of how efficiently pairwise dependent events may be scheduled when scheduling is prone to errors or influence from uncooperative parties. From this perspective, Alice's moves represent an attempt to schedule events efficiently, and Bob's moves represent a worst-case error scenario. Computing $\chi_g(G)$ for an arbitrary graph $G$ is PSPACE-hard, making $\chi_g(G)$ more difficult to compute than the traditional chromatic number $\chi(G)$, unless NP$=$PSPACE \cite{Bodlaender}. However, bounds on the game chromatic number are known for several classes of graphs, including planar graphs \cite{ZhuRefined}, graphs of bounded genus, graphs of bounded treewidth \cite{ZhuKtrees}, and graphs of bounded acyclic chromatic number \cite{DinskiZhu}.

Related to the coloring game is the \emph{marking game}, first considered by U. Faigle et al. in \cite{KiersteadColoring} and first treated formally by X. Zhu in \cite{ZhuColoring}. In the marking game, Alice and Bob take turns marking vertices of a graph, with Alice marking first. In a play of the marking game on a graph $G$, each vertex $v$ receives a score equal to the number of neighbors of $v$ marked before $v$, and that play of the marking game receives an overall score equal to the maximum score of a vertex. Alice's goal in the marking game is to minimize the play score, and Bob's goal is to maximize the play score. The \emph{game coloring number} of a graph $G$, written $\col_{g}(G)$, is the minimum integer $k$ for which Alice has a strategy in the marking game on $G$ to limit the number of marked neighbors of each unmarked vertex at each game state to $k-1$---that is, $\col_{g}(G) - 1$ is the minimum play score that Alice can guarantee in the marking game on $G$. The game coloring number of a graph is related to the coloring number $\col(G)$ of a graph $G$, which is defined by Erd\H{o}s and Hajnal as one more than a graph's degeneracy \cite{ErdosColoring}. The game coloring number satisfies the following inequality:
$$\col(G) \leq \col_g(G) \leq \Delta(G) + 1.$$
It is also straightforward to see that $\chi(G) \leq \col(G)$ and that $\chi_g(G) \leq \col_g(G)$, as the coloring number and game coloring number of a graph are calculated by assuming a ``worst-case scenario" for an uncolored vertex $v$ at each partial coloring, in which all neighbors of $v$ are colored with different colors. Zs. Tuza and X. Zhu give a survey of results related to graph coloring games in \cite{ZhuTuza}.

Recently, Charpentier et al. introduced the notion of the \emph{connected coloring game} and \emph{connected marking game} for a graph $G$, which have the same rules as their traditional counterparts with the exception that Alice and Bob must play so that the set of played vertices always forms a connected set \cite{Charpentier}. These connected games give rise to the graph parameters of \emph{connected game chromatic number} and \emph{connected game coloring number}, written respectively as $\chi_{cg}(G)$ and $\col_{cg}(G)$, which are respectively defined in the same way as $\chi_g(G)$ and $\col_g(G)$, but with respect to the connected coloring game and connected marking game. Similarly to $\chi_g(G)$ and $\col_{g}(G)$, these connected parameters satisfy the following inequalities:
$$\chi(G) \leq \chi_{cg}(G) \leq \Delta(G) + 1;$$
$$\col(G) \leq \col_{cg}(G) \leq \Delta(G) + 1;$$
$$\chi_{cg}(G) \leq \col_{cg}(G).$$

Charpentier et al. also observe that the inequalities $\chi_{cg}(G) \leq \chi_g(G)$ and $\col_{cg}(G) \leq \col_g(G)$ hold for several classes of graphs, including bipartite graphs and outerplanar graphs, with the inequality often strict. However, it is still unknown whether there exists a graph for which $\chi_{cg}(G) > \chi_g(G)$ or $\col_{cg}(G) > \col_g(G)$ holds; that is, it is unknown whether the connected game restriction might help Bob on certain graphs.

Charpentier et al. show that if $G$ is bipartite, then $\chi_{cg}(G) = 2$, and if $G$ is outerplanar, then $\col_{cg}(G) \leq 5$. Apart from these bounds, little else is known about these connected game graph parameters. We will establish bounds for the connected game coloring number of graphs of bounded treedepth and for $k$-trees. Furthermore, we will show that there exists an outerplanar $2$-tree with connected game chromatic number $5$, which answers a question from \cite{Charpentier}.

\section{Graphs of bounded treedepth}
In this section, we consider graphs of bounded treedepth. We define treedepth as follows. Given a rooted tree $T$, we define the \emph{height} of $T$ as the number of vertices in the longest path $P \subseteq T$ with the root of $T$ as an endpoint. We furthermore define the closure of $T$ as the graph on $V(T)$ constructed by adding an edge from each vertex $v \in V(T)$ to all ancestors of $v$ and all descendants of $v$ with respect to the rooted structure of $T$. Then, for a graph $G$, we define the \emph{treedepth} of $G$ as the minimum integer $k$ for which there exists a rooted tree $T$ of height $k$ such that $G$ is a subgraph of the closure of $T$. It is straightforward to show that if $G$ is a connected graph of treedepth $k$ and $G$ is a subgraph of the closure of a rooted tree $T$ of height $k$, then the root of $T$ must be a vertex of $G$. Treedepth is of particular interest in computational applications of graph theory, as several intractible problems, such as graph isomorphism \cite{Bouland} and coloring reachability \cite{Wrochna}, are tractable on graphs of bounded treedepth.

It is easy to see that a graph $G$ of treedepth $1$ has $\col_g(G) = \col_{cg}(G) = 1$ and that a graph $G$ of treedepth $2$ has $\col_g(G) = \col_{cg}(G) = 2$; therefore, we will consider graphs of treedepth at least $3$. We will show that for a graph of treedepth $k \geq 3$, $\col_g(G) \leq 2k - 3$ and $\col_{cg}(G) \leq 2k - 3$, and we will see that both of these bounds are best possible.

\begin{theorem}
Let $k \geq 3$ be an integer, and let $G$ be a connected graph of treedepth $k$. Then $\col_g(G) \leq 2k - 3$, and $\col_{cg}(G) \leq 2k - 3$.
\label{thmConTD}
\end{theorem}
\begin{proof}
Let $G$ be a graph, and let $T$ be a rooted tree of height $k$ such that $G$ is a subgraph of the closure of $T$. Let $r \in V(G)$ be the root of $T$. For a vertex $v \in V(G)$, we will refer to \emph{ancestors} and \emph{descendants} of $v$ with respect to the rooted structure of $T$. Furthermore, we will refer to the \emph{level} of $v$ as the distance from $r$ to $v$ in $T$. We describe a marking game strategy on $G$ by which Alice can guarantee that any unmarked vertex of $G$ has at most $2k-4$ unmarked neighbors. Alice will only mark vertices that are adjacent to already marked vertices, but we will give Bob the freedom to mark any vertex that he wishes. At any point in the game, if an unmarked vertex $v$ is adjacent to a marked vertex, then we say that $v$ is \emph{available}.

Alice begins the game by marking $r$. Then, for each vertex $v \in V(G)$ that Bob marks, Alice marks the vertex $u$ of least level such that $u$ is an ancestor of $v$ and $u$ is unmarked and available. If no such vertex $u$ exists, then Alice arbitrarily chooses an available vertex $w \in V(G)$ of least level and marks $w$. Alice continues this strategy until all vertices of $G$ are marked. 

We claim that at any point in the game, for any unmarked vertex $v \in V(G)$, Alice has marked at most one descendant neighbor of $v$. Indeed, suppose that on one of Alice's turns before $v$ is marked, at least one descendant neighbor of $v$ has already been marked. If Bob has not just marked a descendant of $v$, then Alice's strategy will not consider any descendant of $v$ for marking. If Bob has just marked a descendant of $v$, then as $v$ is available to Alice on this turn, Alice will choose to mark $v$ over any descendant of $v$. Therefore, we see that Alice marks at most one descendant neighbor of $v$ before $v$ is marked, and we see furthermore that in this case, the descendant neighbor of $v$ that Alice marks must be the first marked descendant neighbor of $v$ of the entire game. 

Next, we claim that that each step of the game, for any unmarked vertex $v \in V(G)$, at most $k - 2$ descendant neighbors of $v$ are marked by Bob. To prove our claim, we show that after each of Bob's turns, the number of descendant neighbors of $v$ marked by Bob is at most the number of marked ancestors of $v$. This is certainly true after the first turn.  Furthermore, each time Bob marks a descendant neighbor $v'$ of $v$, Alice marks the ancestor $w$ of $v'$ such that $w$ is available and of least level. When this happens, $w$ must either be equal to $v$ or be an ancestor of $v$ by the fact that $v$ is an available vertex. Therefore, each time that Bob increases the number of marked descendant neighbors of $v$ by one, Alice either increases the number of marked ancestors of $v$ by one or marks $v$.  Therefore, as a non-leaf vertex has at most $k - 2$ ancestors, the claim holds by an inductive argument. 

Now, we calculate the maximum number of marked neighbors of an unmarked vertex $v \in V(G)$. If $v$ is a leaf of $T$, then $v$ has at most $k - 1 < 2k - 3$ marked neighbors. Otherwise, $v$ has at most $k - 2$ ancestors. If Alice did not mark a descendant neighbor of $v$, then $v$ has at most $k-2$ marked ancestors and at most $k-2$ marked descendant neighbors, for a total of at most $2k - 4$ marked neighbors. If Alice marked a descendant neighbor of $v$, then $vr \not \in E(G)$; otherwise, $v$ would be available for Alice to mark after the first move, and Alice would mark $v$ before any descendant of $v$. Hence, in this case, $v$ has at most $k - 3$ marked ancestor neighbors and at most $k - 1$ marked descendant neighbors, for a total of at most $2k - 4$ marked neighbors. 

As Alice's strategy obeys the connected marking game's restriction, this shows that $\col_{cg}(G) \leq 2k - 3$. As Bob's strategy is not required to obey the connected marking game's restriction, this also shows that $\col_g(G) \leq 2k - 3$.
\end{proof}

We will show that these bounds are best possible. In fact, we will show that Bob has a strategy in the graph coloring game to guarantee that a certain graph of treedepth $k$ may not be properly colored with fewer than $2k - 4$ distinct colors. Bob's strategy only requires him to color vertices that are adjacent to vertices that have already been colored, so Bob's strategy will apply to both the connected game and the non-connected game. 

\begin{theorem}
Let $k \geq 3$ be an integer. There exists a connected graph $G$ of treedepth $k$ such that $\chi_g(G) = \col_g(G) = 2k - 3$, $\chi_{cg}(G) = \col_{cg}(G) = 2k - 3$.
\end{theorem}
\begin{proof}
Consider a path $P = (p_1, p_2, \dots, p_{k-1})$ and a tree $T$ with root $p_1$ obtained by attaching $2k-4$ leaves $l_1, \dots, l_{2k-4}$ to $p_{k-1}$. Let $G$ then be the closure of $T$. Clearly $G$ has treedepth $k$. As $\omega(G) = \chi(G) = k$, Alice needs at least $k$ colors to win the game. We will show that if Alice and Bob play the (connected) graph coloring game with $t$ ($k \leq t \leq 2k - 4$) colors, then Bob has a winning strategy.

If Alice first colors a vertex of $P$ with a color (say $1$), then Bob colors the leaf $l_1$ with color $2$. If Alice first colors a leaf $l_i$ with a color (say $1$), then Bob colors $p_1$ with the color $2$. After this first move, Bob's strategy will be to color as many leaves $l_i$ as possible with distinct colors other than $1$ and $2$. Note that as some vertex of $P$ is colored after the first turn, all leaves $l_i$ will subsequently be available to Bob in the connected game. As there are $2k - 4$ leaves $l_i$, Bob will either succeed in letting at least $k - 2$ leaves $l_i$ be colored with distinct colors before all vertices of $P$ are colored, or Bob will reach a position in which each of the $t$ colors appears either at a vertex of $P$ or at a leaf $l_i$, in which case no more vertices of $P$ may be properly colored, and Bob wins the game. Hence we assume that Bob lets $k-2$ leaves $l_i$ be colored with distinct colors.

Suppose, for the sake of contradiction, that Alice wins the game. This implies that a proper coloring on $G$ is completed. Consider the state of the game immediately before the last vertex $v$ of $P$ is colored. As $v \in V(P)$, $v$ is adjacent to $k-2$ other vertices of $P$ and all leaves $l_i$. Additionally, as $V(P)$ induces a clique in $G$, each of these $k-2$ neighbors of $v$ in $V(P)$ have distinct colors. Furthermore, $v$ is adjacent to leaves $l_i$ of at least $k-2$ additional distinct colors, as each vertex of $V(P)$ of adjacent to each leaf $l_i$, these $k-2$ additional distinct colors do not coincide with the $k-2$ colors of $V(P)$. Therefore, altogether, the neighborhood of $v$ contains vertices of at least $2k - 4$ distinct colors, which contradicts the assumption that $v$ was colored successfully. Therefore, $\chi_g(G) = 2k - 3$, and $\chi_{cg}(G) = 2k - 3$.
\end{proof}

\section{$k$-trees}
In this section, we will consider $k$-trees. A $k$-tree is defined as a graph that may be constructed by starting with a copy of $K_k$ and then iteratively adding new vertices of degree $k$ whose neighbors induce a $K_k$. We will show that for any $k$-tree $G$, $\col_{cg}(G) \leq 3k$, and we will show that for any $2$-tree $G$, $\col_{cg}(G) \leq 5$. X. Zhu shows that for a partial $k$-tree $G$, which is defined as a subgraph of a $k$-tree, $\col_g(G) \leq 3k + 2$ \cite{ZhuKtrees}, and J. Wu and X. Zhu show furthermore that this bound is tight \cite{WuLower}. Therefore, our bound shows that the connected game condition strictly helps Alice in the marking game on $k$-trees.

In \cite{Charpentier}, Charpentier et al. ask if $\col_{cg}(H) \leq \col_{cg}(G)$ holds for every subgraph $H$ of a graph $G$. If this inequality is true, then this would imply that our upper bounds on the connected game coloring number of $k$-trees would hold for partial $k$-trees. In particular, this would show that partial $2$-trees have connected game coloring number at most $5$, which would generalize a result from \cite{Charpentier} stating that outerplanar graphs have a connected game coloring number of at most $5$. 

In order to prove our results, we will use the \emph{activation strategy}, which was first introduced by X. Zhu in \cite{ZhuKtrees}. For a concise description of the strategy, we refer the reader to \cite{ZhuTuza}. Our main results of this section will show that in the connected marking game on $k$-trees, the activation strategy gives a better upper bound on the game coloring number than in the non-connected marking game, and that the activation strategy is at times best possible for $2$-trees.

We will consider a fixed $k$-tree $G$, and we will give $V(G)$ an ordering $v_1, \dots, v_n$ such that $v_1, \dots, v_k$ induce a clique in $G$ and such that for any $v_i, i \geq k + 1$, the neighbors $v_j$ of $v_i$ with $j < i$ induce a $k$-clique in $G$. For a vertex $v_i$, we say that a vertex $v_j \in N(v_i)$ with $j < i$ is a \emph{back-neighbor} of $v_i$. We say that all other neighbors of $v_i$ are \emph{fore-neighbors}.

We describe Zhu's activation strategy for the sake of completeness. Throughout the game, Alice will construct a digraph $D$ for which $V(D) = V(G)$ that will help her decide which moves to play. Initially, $D$ has no arcs, and as each player takes turns, Alice will add arcs to $D$. For a vertex $v \in V(D)$, we let $d^{-}(v)$ refer to the in-degree of $v$ in $D$. On each of her turns, Alice will \emph{process} a vertex. When we say that Alice processes a vertex $v_i \in V(G)$, we mean that Alice executes the following procedure:
\begin{enumerate}
\item If $v_i$ has no unmarked back-neighbor and $v_i$ is unmarked, then Alice marks $v_i$ and stops.
\item If $v_i$ has no unmarked back-neighbor and $v_i$ is marked, then Alice marks any vertex with no unmarked back-neighbor and stops.
\item Otherwise, let $v_j$ be the unmarked back-neighbor of $v_i$ with smallest index $j$. If $d^{-}(v_j) \geq 1$, then Alice marks $v_j$, adds the arc $v_i v_j$ to $D$, and stops. If $d^{-}(v_j) = 0$, then Alice adds the arc $v_i v_j$ to $D$, and then processes $v_j$.
\end{enumerate}
Now, with this definition in place, Alice's strategy is as follows. Alice begins the game by marking $v_1$. Then, when Bob marks a vertex $v_i$, Alice processes $v_i$.

We say that a vertex is \emph{active} if it is marked or if it is incident to an arc of $D$. For vertices $v,w \in V(G)$, we say that $w$ is \emph{reachable} from $v$ if there exists a directed path in $D$ from $v$ to $w$. We will need a lemma about $D$.

\begin{lemma}
Consider a state of a connected marking game on $G$. Let $C \subseteq G$ be the connected subgraph of $G$ induced by marked vertices. For any vertex $v_i \in C$, every vertex $v_j$ reachable from $v_i$ either belongs to $C$ or is adjacent to a vertex of $C$.
\label{lemmaLegal}
\end{lemma}
\begin{proof}
As Alice begins the game by marking $v_1$, we assume that $v_1 \in C$. We induct on the distance $d$ from $v_i$ to $v_j$ in $D$. The statement clearly holds for $d = 0$. For $d \geq 1$, consider a vertex $v_j$ reachable in $D$ from $v_i$ in $d$ steps. To reach $v_j$ from $v_i$, there must exist a vertex $v_l$ at distance $d-1$ from $v_i$ such that $v_j$ is an out-neighbor of $v_l$. As the back-neighbors of $v_l$ either contain $v_1$ or form a cut separating $v_l$ and its fore-neighbors from $v_1$, it must follow that some back-neighbor of $v_l$ is marked; otherwise, $C$ would not be a connected set. Furthermore, all of the back-neighbors of $v_l$ form a clique, and therefore, $v_j$ either belongs to $C$ or is adjacent to a vertex of $C$. This completes induction.
\end{proof}

We are now ready to establish our upper bound for the connected game coloring number of $k$-trees.
\begin{theorem}
Let $k \geq 2$, and let $G$ be a $k$-tree. Then $\col_{cg}(G) \leq 3k$. Furthermore, if $k = 2$, then $\col_{cg}(G) \leq 5$.
\end{theorem}
\begin{proof}
We let Alice follow the activation strategy as described above. We show that this strategy limits the number of marked neighbors of any unmarked vertex to $3k-1$ when $k \geq 3$, and to $4$ when $k \geq 2$.

First, we will show that Alice's strategy is a legal strategy; that is, we show that apart from the first move, Alice always marks a vertex that is adjacent to an already marked vertex. If Alice marks a vertex $v_i$ by either (1) or (2) of the processing procedure, then $v_i$ certainly has a marked neighbor. If Alice marks a vertex $v_j$ by (3) of the processing procedure, then $d^-(v_j) = 1$ must have held before Alice's move. This implies that $v_j$ is reachable from some vertex marked by Bob, and so by Lemma \ref{lemmaLegal}, $v_j$ is adjacent to a marked vertex.


We will show by induction on the number of moves that Alice has a strategy to ensure that after each of her moves, each unmarked vertex $v_j$ satisfies $d^-(v_j) \leq 1$. This condition certainly holds after Alice's first move. Suppose that the condition holds at some point of the game after Alice's move. On Bob's subsequent move, Bob marks a vertex $v_i$, and Alice processes $v_i$. During the processing procedure, if at any point Alice adds an arc to $D$ that causes a vertex $v_j$ to have $d^-(v_j) = 2$, then the processing procedure terminates, and Alice marks $v_j$. As each new arc of $D$ can only increase the in-degree of one vertex, it must follow in this case that $v_j$ is the only vertex with $d^-(v_j) = 2$. Therefore, after Alice marks $v_j$, each unmarked vertex satisfies has in-degree at most $1$. Furthermore, as the in-degree of a marked vertex may not increase, the same argument shows that each vertex of $D$ has in-degree at most $2$.

We now consider an unmarked vertex $v$ at some point $(*)$ of the game and compute an upper bound for the number of neighbors of $v$ that are marked at $(*)$. If $v$ has no marked fore-neighbor at $(*)$, then $v$ has at most $k$ marked neighbors, and the theorem follows. Otherwise, let $B$ be the set of back-neighbors of $v$ that were unmarked immediately before the first fore-neighbor of $v$ was marked. As the back-neighbors of $v$ either contain $v_1$ or separate $v$ from $v_1$, at least one back-neighbor of $v$ must be marked before any fore-neighbor of $v$ is marked, and hence $|B| \leq k - 1$. Now, suppose a fore-neighbor $w$ of $v$ is marked at some point in the game before $(*)$. If $w$ is marked by Bob, then Alice processes $w$ and adds an arc $wx$ to $D$, where $x \in \{v\} \cup B$. If $w$ is marked by Alice, then as $w$ has an unmarked back-neighbor, Alice must have marked $w$ because $d^-(w) = 2$. This implies that before Alice began her move, $d^-(w) = 1$, and as $w$ was unmarked and had an unmarked back-neighbor, there exists an arc $wx$ in $D$, where $x \in \{v\}\ \cup B$. 

Finally, we observe that if $B$ is empty, then as there exists an arc $wv$ for each marked fore-neighbor of $v$, $v$ has at most two marked fore-neighbors. As $v$ has at most $k$ marked back-neighbors, the total number of marked neighbors of $v$ at the point $(*)$ is $k+2$, which proves the theorem for all values $k \geq 2$. Otherwise, $|B| \geq 1$, and so when $d^-(v)$ first increased to $1$ (if ever), Alice must have added an arc $vb$ to $D$ for some $b \in B$. As each vertex of $\{v\} \cup B$ has in-degree at most $2$, this altogether implies that $v$ has at most $2|\{v\} \cup B| - 1 \leq 2k - 1$ marked fore-neighbors at $(*)$. As $v$ has at most $k$ back-neighbors, $v$ altogether has at most $3k - 1$ marked neighbors at $(*)$, which completes the proof for $k \geq 3$. 


If $k = 2$, we may obtain a slightly better upper bound by noting that for a fore-neighbor $w$ of $v$, $v$ is the only unmarked back-neighbor of $w$, and hence for every marked fore-neighbor $w$ of $v$, there exists an arc $wv$ in $D$. As $d^-(v) \leq 2$, this implies that $v$ has at most two marked fore-neighbors at $(*)$, for a total of at most four marked neighbors altogether. This completes the proof for the case that $k = 2$. 
\end{proof}

Finally, we show that for a $2$-tree $G$, an upper bound of $5$ for $\col_{cg}(G)$ is best possible. Furthermore, we show that this upper bound may not be improved by restricting ourselves to outerplanar $2$-trees. This answers a question of Charpentier et al. in \cite{Charpentier} which asks if $5$ is the best possible upper bound for the connected game coloring number of outerplanar graphs.

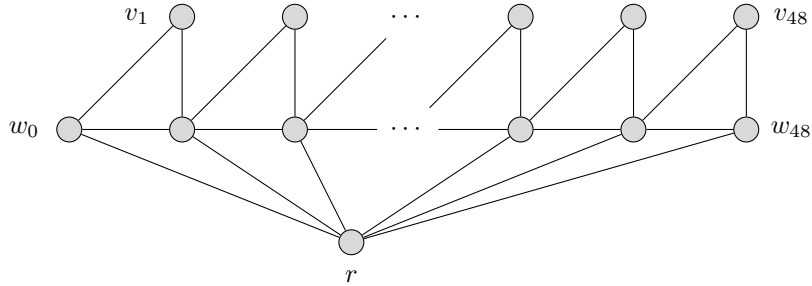
\begin{figure}
\begin{tikzpicture}
[scale=1.5,auto=left,every node/.style={circle,fill=gray!30}]

\node (z) at (2.5,0.7) [fill = white] {$r$};
\node (z) at (0.6,3) [fill = white] {$v_1$};
\node (z) at (6.4,3) [fill = white] {$v_{48}$};
\node (z) at (-0.4,2) [fill = white] {$w_0$};
\node (z) at (6.4,2) [fill = white] {$w_{48}$};

\node (r) at (2.5,1) [draw = black] {};
\node (a1) at (0,2) [draw = black] {};
\node (a2) at (1,2) [draw = black] {};
\node (a3) at (2,2) [draw = black] {};
\node (a4) at (3,2) [fill = white] {$\cdots$};
\node (a5) at (4,2) [draw = black] {};
\node (a6) at (5,2) [draw = black] {};
\node (a7) at (6,2) [draw = black] {};
\node (b2) at (1,3) [draw = black] {};
\node (b3) at (2,3) [draw = black] {};
\node (b4) at (3,3) [fill = white] {$\cdots$};
\node (b5) at (4,3) [draw = black] {};
\node (b6) at (5,3) [draw = black] {};
\node (b7) at (6,3) [draw = black] {};

\foreach \from/\to in {r/a1,r/a2,r/a3,r/a5,r/a6,r/a7,a1/b2,a2/b3,a3/b4,a4/b5,a5/b6,a6/b7,a1/a2,a2/a3,a3/a4,a4/a5,a5/a6,a6/a7,a2/b2,a3/b3,a5/b5,a6/b6,a7/b7}
    \draw (\from) -- (\to);
\end{tikzpicture}
\caption{The depicted graph $G$ is an outerplanar $2$-tree on $98$ vertices. The vertices in the middle row from left to right are $w_0, w_1, \dots, w_{48}$. The vertices in the top row from left to right are $v_1, \dots, v_{48}$.} 
\label{fig:5}
\end{figure}

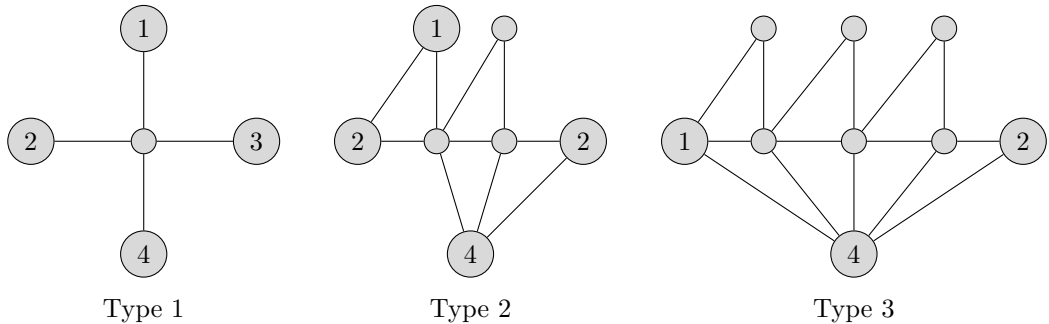
\begin{figure}
\begin{tikzpicture}
[scale=1.5,auto=left,every node/.style={circle,fill=gray!30}]
\node (z) at (1,0.5) [fill = white] {Type 1};
\node (z) at (2.5,0.5) [fill = white] {};
\node (r) at (1,1) [draw = black] {$4$};
\node (a1) at (0,2) [draw = black] {$2$};
\node (a2) at (1,2) [draw = black] {};
\node (a3) at (2,2) [draw = black] {$3$};
\node (b2) at (1,3) [draw = black] {$1$};
\foreach \from/\to in {r/a2,a2/b2,a1/a2,a2/a3}
    \draw (\from) -- (\to);
\end{tikzpicture}
\begin{tikzpicture}
[scale=1.5,auto=left,every node/.style={circle,fill=gray!30}]
\node (z) at (1.5,0.5) [fill = white] {Type 2};
\node (z) at (3,0.5) [fill = white] {};
\node (r) at (1.5,1) [draw = black] {$4$};
\node (a1) at (0.5,2) [draw = black] {$2$};
\node (a2) at (1.2,2) [draw = black] {};
\node (a3) at (1.8,2) [draw = black] {};
\node (a4) at (2.5,2) [draw = black] {$2$};
\node (b2) at (1.2,3) [draw = black] {$1$};
\node (b3) at (1.8,3) [draw = black] {};
\foreach \from/\to in {r/a2,r/a3,r/a4,a2/b2,a3/b3,a1/a2,a2/a3,a3/a4,a2/b3,a1/b2}
    \draw (\from) -- (\to);
\end{tikzpicture}
\begin{tikzpicture}
[scale=1.5,auto=left,every node/.style={circle,fill=gray!30}]
\node (z) at (2,0.5) [fill = white] {Type 3};
\node (r) at (2,1) [draw = black] {$4$};
\node (a1) at (0.5,2) [draw = black] {$1$};
\node (a2) at (1.2,2) [draw = black] {};
\node (a3) at (2,2) [draw = black] {};
\node (a4) at (2.8,2) [draw = black] {};
\node (a5) at (3.5,2) [draw = black] {$2$};
\node (b2) at (1.2,3) [draw = black] {};
\node (b3) at (2,3) [draw = black] {};
\node (b4) at (2.8,3) [draw = black] {};
\foreach \from/\to in {r/a1,r/a2,r/a3,r/a4,a2/b2,a3/b3,a1/a2,a2/a3,a3/a4,a1/b2,a2/b3,a5/a4,r/a5,a3/b4,b4/a4}
    \draw (\from) -- (\to);
\end{tikzpicture}
\caption{The figure shows three graphs whose vertices are partially colored by the set $\{1,2,3,4\}$. Bob wins the connected coloring game on $G$ played with four colors whenever one of these three types of graphs is isomorphic as a partially colored graph to a subgraph of $G$.} 
\label{figTypes}
\end{figure}

\begin{theorem}
There exists an outerplanar $2$-tree $G$ such that $ \chi_{cg}(G) = \col_{cg}(G) = 5$.
\end{theorem}
\begin{proof}
We will consider the graph $G$ from Figure \ref{fig:5}, which is an outerplanar $2$-tree on $98$ vertices. We will consider a connected coloring game between Alice and Bob on $G$ and show that Alice may not win with fewer than five colors. Bob will win if the game is played with two colors, as $G$ is not bipartite. If the game is played with three colors, then as $G$ is uniquely $3$-colorable, there exists a unique coloring up to permutation with which Alice may win. In this unique coloring, all vertices $v_i$ must receive the same color. It is not difficult to show that Alice cannot prevent Bob from coloring two distinct vertices $v_i$ and $v_j$ with different colors and hence winning the game. 

Finally, we consider the connected coloring game played with four colors. We will first show that if $G$ has a subgraph that is isomorphic as a partially colored graph to one of the graphs in Figure \ref{figTypes}, then Bob has a winning strategy. In demonstrating this claim, we will assume that if $G$ has a subgraph $H$ isomorphic to one of the types of Figure \ref{figTypes}, then Alice and Bob will only color vertices of $H$. In fact, Alice may continue to color vertices outside of $H$ and thereby ``pass" in the coloring game on $H$, but it is straightforward to check in the following arguments that Alice gains no advantage by refusing to color a vertex from $H$.

\begin{claim}
Suppose that at some point of the game, $G$ contains a subgraph that is isomorphic as a partially colored subgraph to the graph of Type 1 in Figure \ref{figTypes}. Then Bob has a strategy to win the game.
\end{claim}
\begin{proof}
This claim is trivial.
\end{proof}
\begin{claim}
Suppose that at some point of the game, $G$ contains a subgraph that is isomorphic as a partially colored subgraph to the graph of Type 2 in Figure \ref{figTypes}. Then Bob has a strategy to win the game. 
\end{claim}
\begin{proof}
Suppose the graph of Type $2$ occurs as a partially colored subgraph of $G$. This graph has three uncolored vertices, which we refer to as North, West, and Southeast. If it is Bob's turn, then Bob may win by coloring Southeast with $3$, which gives a subgraph of Type $1$. 

Suppose that it is Alice's move. By the connectivity condition, Alice must color either West or Southeast. If Alice colors Southeast with $1$, then Bob may obtain a subgraph of Type 1 by coloring North with $3$. If Alice colors West with $3$, then Bob may color North with $1$ and again obtain a subgraph of Type $1$. If Alice colors Southeast with $3$, then Bob wins immediately. Therefore, Bob has a strategy to win the game. 
\end{proof}
\begin{claim}
Suppose that at some point of the game, $G$ contains a subgraph that is isomorphic as a partially colored subgraph to the graph of Type 3 in Figure \ref{figTypes}. Then Bob has a strategy to win the game. 
\end{claim}
\begin{proof}
The graph of Type $3$ has six uncolored vertices. We refer to the upper three vertices, from left to right, as $a_1, a_2, a_3$. We refer to the lower three vertices, from left to right, as $b_1, b_2, b_3$.

If it is Bob's turn, then Bob colors $a_1$ with $3$. If Alice colors $b_1$ with $2$, then Bob creates a subgraph of Type 2 by coloring $a_2$ with $1$. If Alice colors $b_2$ with $3$ or $1$, then Bob creates a subgraph of Type 1 by coloring $a_3$ respectively with $1$ or $3$. Any other move that Alice plays will either create a subgraph of Type 1 or allow Bob to create a subgraph of Type 1 by coloring a vertex with $2$ on the next turn.  

Suppose, on the other hand, that it is Alice's turn. If Alice plays $2$ or $3$ at $a_1$, the Bob wins with a Type 1 subgraph by playing a $3$ or $2$ respectively at $b_2$. If Alice plays $4$ at $a_1$, then Bob plays $3$ at $b_2$, threatening to play $2$ at $a_2$ and to play $1$ at $a_3$, both of which are a win for Bob by a Type 1 subgraph. Alice cannot address both of these threats, so Bob wins. 

If Alice plays $2$ at $b_1$, then Bob creates a Type 2 subgraph by playing $1$ at $a_2$. If Alice plays $3$ at $b_1$, then Bob plays $2$ at $a_2$, again creating a subgraph of Type 2.

If Alice colors $b_2$ with $1$ or $3$, then Bob wins by coloring $a_3$ with $3$ or $1$, respectively. If Alice colors $b_2$ with $2$, then Bob wins by coloring $a_1$ with $3$. 

If Alice colors $b_3$ with $1$ or $3$, then Bob creates a Type $2$ subgraph by coloring $a_1$ with $3$. 
\end{proof}
We now describe Bob's strategy. First, Bob uses at most two moves to ensure that $r$ is colored. We assume without loss of generality that $r$ is colored with the color $4$. Then, on Bob's first move after $r$ is colored, at most four vertices apart from $r$ are colored, and hence there must exist a value $i, 0 \leq i \leq 40$ such that $w_{i}, w_{i+1}, \dots, w_{i+8}$ and $v_{i+1}, \dots, v_{i+7}$ are all uncolored. Bob will color $w_{i+4}$ with the color $1$. On Bob's next turn, Bob will be able to use color $2$ to color either $w_i$ or $w_{i+8}$ and create a subgraph in $G$ isomorphic as a partially colored graph to the graph of Type 3 from Figure \ref{figTypes}. Thus Bob has a strategy to win the game.
\end{proof}
We conclude the section by noting that when the aim is only to show that $\col_{cg}(G) = 5$, Bob has a much simpler strategy, and $G$ can be made much smaller while still satisfying $\col_{cg}(G) = 5$. We encourage the reader to find a simplified strategy for Bob in the connected graph marking game on $G$, as well as a way to reduce the size of $G$ while still letting $G$ satisfy $\col_{cg}(G) = 5$.

\raggedright
\bibliographystyle{abbrv}
\bibliography{connectedBib}

\end{document}